\documentclass{amsart}
\usepackage{tikz-cd}
\usepackage{amsfonts,graphicx,rotating,ctable}
\usepackage{amssymb}
\usepackage[hidelinks]{hyperref}
\usepackage{changepage}
\usepackage{mathtools}
\usepackage{comment}
\usepackage{relsize}

\input diagxy

\makeatletter
\@namedef{subjclassname@2020}{\textup{2020} Mathematics Subject Classification}
\makeatother

\newtheorem{theorem}{Theorem}[section]
\newtheorem{lemma}[theorem]{Lemma}
\newtheorem{corollary}[theorem]{Corollary}
\newtheorem{proposition}[theorem]{Proposition}

\theoremstyle{definition}
\newtheorem{definition}[theorem]{Definition}

\theoremstyle{remark}
\newtheorem{remark}[theorem]{Remark}

\numberwithin{equation}{section}


\newcommand{\ges}{\geqslant}

\newcommand{\R}{\mathbb R}

\newcommand{\F}{\mathbb F}

\newcommand{\gr}{\widetilde G}
\newcommand{\w}{\widetilde w}

\newcommand{\lm}{\operatorname{LM}}

\newcommand{\htt}{\operatorname{ht}}

\newcommand{\cl}{\operatorname{cl}}
\newcommand{\zcl}{\operatorname{zcl}}
\newcommand{\cat}{\operatorname{cat}}
\newcommand{\tc}{\operatorname{TC}}
\newcommand{\arr}{\rightarrow}

\renewcommand{\le}{\leqslant}
\renewcommand{\ge}{\geqslant}

\begin{document}

\title[Heights of SW-classes and zcl of Grassmannians]{Heights of Stiefel--Whitney classes and zero-divisor cup-length of some Grassmann manifolds}


\author[M. Jovanovi\' c]{Milica Jovanovi\' c}
\address{University of Belgrade,
  Faculty of mathematics,
  Studentski trg 16,
  Belgrade,
  Serbia}
\email{milica.jovanovic@matf.bg.ac.rs}

\author[V. Ovaskainen]{Vuk Ovaskainen}
\address{University of Belgrade,
  Faculty of mathematics,
  Studentski trg 16,
  Belgrade,
  Serbia}
\email{vuk.simon.ovaskainen@matf.bg.ac.rs}

\author[B. I. Prvulovi\' c]{Branislav I.\ Prvulovi\'c}
\address{University of Belgrade,
  Faculty of mathematics,
  Studentski trg 16,
  Belgrade,
  Serbia}
\email{branislav.prvulovic@matf.bg.ac.rs}

\author[A. Suboti\' c]{Antonije Suboti\' c}
\address{University of Belgrade,
  Faculty of mathematics,
  Studentski trg 16,
  Belgrade,
  Serbia}
\email{antonije.subotic@matf.bg.ac.rs}

\thanks{The authors are partially supported by the Ministry of Science, Technological Development and Innovation, Republic of Serbia, through the project 451-03-33/2026-03/200104.} 



\begin{abstract}
We calculate the heights of Stiefel--Whitney classes of the canonical vector bundle over the oriented Grassmannians $\gr_{n,4}\cong SO(n)/(SO(4)\times SO(n-4))$ in the cases $n\in\{2^t-2,2^t-1,2^t,2^t+1\}$, $t\ge4$. Using some additional computations in modulo $2$ cohomology of $\gr_{n,4}$ and the well-known connection between topological complexity and zero-divisor cup-length, we obtain lower bounds for topological complexity of these Grassmannians. We also extend recent results of Rusin, who computed the modulo $2$ cup-length of $\gr_{n,4}$ for $n\in\{2^t-2,2^t-1,2^t\}$, to the case $n=2^t+1$, $t\ge3$.
\end{abstract}

\maketitle

\section{Introduction}
\label{sec:introduction}

A classical invariant of a topological space is the Lusternik--Schnirelmann category. For a (reasonable) space $X$, it is defined as the minimal number $d$ such that there exists an open cover of $X$ consisting of $d$ open subsets each of which is contractible in $X$. An equivalent definition states that the Lusternik--Schnirelmann category of a path connected space $X$ is the (unreduced) Schwarz genus of the path-space fibration $p:PX\arr X$, where $PX=\{\omega\in X^I\mid\omega(0)=x_0\}$ is the space of paths in $X$ starting at a base point $x_0\in X$, and $p(\omega)=\omega(1)$. We will denote by $\cat(X)$ the Lusternik--Schnirelmann category of $X$.

If we take the whole space of paths $X^I$ and the fibration $\pi:X^I\arr X\times X$ given by $\pi(\omega)=(\omega(0),\omega(1))$, the Schwarz genus of this fibration is another invariant, called topological complexity of $X$, denoted by $\tc(X)$, and introduced by Farber in \cite{Farber}.

One of the most common approaches for studying these invariants is the cohomological method, which provides lower bounds for them in terms of cup product. A lower bound for $\cat(X)$ is the $R$-cup-length of $X$ (for any commutative ring $R$), denoted by $\cl_R(X)$, and a lower bound for $\tc(X)$ is the $R$-zero-divisor cup-length of $X$, denoted by $\zcl_R(X)$:
\begin{equation}\label{eq:cat>1+cup}
    \cl_R(X)<\cat(X),
\end{equation} 
\begin{equation}\label{eq:TC>1+zcl}
     \zcl_R(X) < \tc(X).
\end{equation}

The Grassmann manifold $\gr_{n,k}$ of oriented $k$-dimensional subspaces in $\R^n$ (the so-called oriented Grassmannian) is a closed manifold of dimension $k(n-k)$. When it comes to the invariants $\cat(\gr_{n,k})$ and $\tc(\gr_{n,k})$, the case $k=1$ is well established, since $\gr_{n,1}$ is just the sphere $S^{n-1}$, and so $\cat(\gr_{n,1})=2$, while $\tc(\gr_{n,1})=2$ if $n$ is even, and $\tc(\gr_{n,1})=3$ if $n$ is odd (by \cite[Theorem 8]{Farber}). The case $k=2$ is also uninteresting, because $\gr_{n,2}$ is a closed, simply connected, symplectic manifold (the symplectic structure comes from the identification of $\gr_{n,2}$ with the complex quadric $Q_{n-2}=\{[z_0,\ldots,z_{n-1}]\in\mathbb C\mathrm P^{n-1}\mid z_0^2+\cdots+z_{n-1}^2=0\}$), and consequently, both Lusternik--Schnirelmann category and topological complexity are known: $\cat(\gr_{n,2})=n-1$, $\tc(\gr_{n,2})=2n-3$ (see e.g.\ \cite[Proposition 2.3]{pavesic}).

For $k\ge3$ these invariants are not known. However, there has been some development recently in calculating cohomological lower bounds from (\ref{eq:cat>1+cup}) and (\ref{eq:TC>1+zcl}), mostly with modulo $2$ coefficients (i.e., with coefficients in the two-element field $\F_2$). In the case $k=3$, the $\F_2$-cup-length of $\gr_{n,3}$ was studied and computed for different values of $n$ in a number of papers (e.g.\ \cite{Fukaya,Korbas}) and in \cite{CP-cl} it was computed for all $n\ge6$ (it suffices to study $\gr_{n,k}$ for $n\ge2k$, because $\gr_{n,n-k}\cong\gr_{n,k}$). The $\F_2$-zero-divisor cup-length of $\gr_{n,3}$ was studied in \cite{CPR}, where $\zcl_{\F_2}(\gr_{n,3})$ is either exactly determined or placed between bounds that differ by $1$.

In the case $k=4$, some lower and upper bounds for $\cl_{\F_2}(\gr_{n,4})$ were established in \cite[Proposition 6.11]{bane-marko}, while the exact value of $\cl_{\F_2}(\gr_{n,4})$ for $n\in\{2^t-2,2^t-1,2^t\}$ ($t\ge4$) was calculated in a recent paper by Rusin \cite[Theorem 1.1]{Rusin}.

In this paper we extend Rusin's results to the case $n=2^t+1$ ($t\ge3$) by computing $\cl_{\F_2}(\gr_{2^t+1,4})$ in Theorem \ref{cor:cl(G2^t+1,4)}, and in Corollary \ref{prop:finale} we establish lower bounds for $\zcl_{\F_2}(\gr_{n,4})$, which, along with the inequality (\ref{eq:TC>1+zcl}), provide lower bounds for topological complexity of these manifolds. As another topic of the paper, which is intimately connected with cup-length and zero-divisor cup-length, we investigate the heights of the Stiefel--Whitney classes of the tautological vector bundle over $\gr_{n,4}$, and determine them completely for $n\in\{2^t-2,2^t-1,2^t,2^t+1\}$ (Theorem \ref{prop:ht} and Theorem \ref{htw_4}).

The main tool in our calculations are Gr\"obner bases for the ideals closely related to the mod $2$ cohomology of $\gr_{n,4}$. These bases for the cases $n\in\{2^t-2,2^t-1,2^t\}$ were obtained by Rusin in \cite{Rusin}. The corresponding ideal for $n=2^t+1$ is actually the same as the one for $n=2^t$, so we have a Gr\"obner basis for each of the cases we consider.

The paper is organized as follows. In Section \ref{sec:2} we give a necessary background for our calculations. We first define and discuss the main notions that we consider in the paper---the height of a cohomology class, the cup-length and the zero-divisor cup-length; then we give a brief review of the theory of Gr\"obner bases over a field; and finally we collect some basic facts concerning the mod $2$ cohomology of oriented Grassmannians $\gr_{n,k}$. Sections \ref{sec:3}, \ref{sec:4} and \ref{sec:5} are devoted to proving our main results, regarding (respectively) heights of Stiefel--Whitney classes, $\F_2$-cup-length and $\F_2$-zero-divisor cup-length.

\section{Preliminaries}
\label{sec:2}

\subsection{Cup-length, zero-divisor cup-length, and height of an algebra element}

Let $R$ be a commutative ring with identity, and $A=\bigoplus_{i=0}^\infty A^i$ a graded $R$-algebra with identity $1\in A^0$. If an element $a\in A$ belongs to the submodule $A^i$, then we say that $a$ is homogeneous of degree $i$, and we write $|a|=i$.

If $a\in A$ is homogeneous of positive degree, we define the {\it height} of $a$, denoted by $\htt(a)$, as the supremum of the set $\{m\mid a^m\neq0\}$.

The {\it cup-length} of the algebra $A$, denoted by $\cl(A)$, is defined as the supremum of the set of all natural numbers $m$ such that there exist homogeneous elements $a_1,a_2,\ldots,a_m\in A$ of positive degree with $a_1\cdot a_2\cdots a_m\neq0$. Less formally, $\cl(A)$ is the length of the longest nontrivial product of positive-degree elements in $A$.

It is obvious that for every $a\in A$ (which is homogeneous of positive degree) one has $\htt(a)\le\cl(A)$.

\medskip

The tensor product $A\otimes A=\bigoplus_{r=0}^\infty\big(\bigoplus_{i=0}^rA^i\otimes A^{r-i}\big)$ becomes a graded $R$-algebra as well, with the product defined on simple tensors (of homogeneous elements $a,b,c,d\in A$) by
\[(a\otimes b)\cdot(c\otimes d)=(-1)^{|b|\cdot|c|}(a\cdot c)\otimes(b\cdot d).\]
Moreover, if algebra $A$ is commutative (in the graded sense---meaning that $b\cdot c=(-1)^{|b|\cdot|c|}c\cdot b$ for any two homogeneous elements $b,c\in A$), then its product map $A\otimes A\stackrel{\cdot}\longrightarrow A$ is an algebra morphism. In that case, the kernel of the map $A\otimes A\stackrel{\cdot}\longrightarrow A$ is an ideal in $A\otimes A$, and its elements are called {\it zero-divisors}. For instance, if $a\in A$ is any element, then
\[z(a)=1\otimes a-a\otimes1\]
is a zero-divisor. Note that if $a$ is homogeneous, then so is $z(a)$, and $|z(a)|=|a|$.

The {\it zero-divisor cup-length} of $A$, denoted by $\zcl(A)$, is the supremum of the set of all natural numbers $m$ such that there are homogeneous zero-divisors $z_1,z_2,\ldots,z_m\in\ker\big(A\otimes A\stackrel{\cdot}\longrightarrow A\big)$ of positive degree with the property $z_1\cdot z_2\cdots z_m\neq0$ in $A\otimes A$. So, one can say that $\zcl(A)$ is the cup-length of $\ker\big(A\otimes A\stackrel{\cdot}\longrightarrow A\big)$, if this kernel is considered as a graded algebra (although without identity).

\medskip

Now, if $X$ is a topological space, then for any commutative ring (with identity) $R$, the cup product of cohomology classes turns $H^*(X;R)=\bigoplus_{i=0}^\infty H^i(X;R)$ into a graded commutative $R$-algebra, and we define {\it $R$-cup-length} and {\it $R$-zero-divisor cup-length} of $X$ as expected:
\[\cl_R(X):=\cl\big(H^*(X;R)\big) \quad \mbox{ and } \quad \zcl_R(X):=\zcl\big(H^*(X;R)\big).\]

\subsection{Gr\"obner bases}
The main purpose of Gr\"obner bases is deciding whether a given polynomial belongs to a given ideal (in a polynomial ring), i.e., whether the coset of that polynomial is zero in the corresponding quotient ring.
Let us recall some basics concerning Gr\"obner bases over a field. For a thorough introduction into the theory one can read \cite[Chapter 5]{Becker}. In order to simplify our discussion, we will present this review for the field $\F_2$ only, since that is the only case we will be working with in this paper.

Let $\F_2[\overline X]$ be the polynomial ring, where $\overline X$ stands for the (finite) set of variables. Denote by $M$ the set of all monomials in $\F_2[\overline X]$. Let us fix a {\it monomial order} $\preceq$, that is, a well ordering on $M$ with the property that for all $m_1,m_2,m_3\in M$, $m_1\preceq m_2$ implies $m_1m_3\preceq m_2m_3$. 

For a polynomial $f = \sum_{i=1}^r m_i\in\F_2[\overline X]$, where $m_i\in M$ are distinct monomials, define $M(f)=\{m_1,\ldots,m_r\}$ (it is understood that $M(0)=\emptyset$). The {\it leading monomial} of a nonzero polynomial $f$ is $\lm(f)\coloneqq\max M(f)$, where the maximum is taken with respect to the ordering $\preceq$.    

Now let $F\subset\F_2[\overline X]$ be a finite set of nonzero polynomials, and denote by $\lm(F)=\{\lm(f)\mid f\in F\}\subseteq M$ the set of leading monomials of the polynomials in $F$. 
\begin{definition}
    For two polynomials $p,q\in\F_2[\overline X]$ we say that $p$ {\it reduces to $q$ modulo $F$}, and write $p\stackrel{F}\longrightarrow q$, if either $p=q=0$ or there exists a finite sequence of polynomials $p=p_0,p_1,\ldots,p_s=q$ with the following property: for every $j\in\{0,1,\ldots,s-1\}$ there is a polynomial $f_j\in F$ such that $\lm(f_j)\mid\lm(p_j)$, say $\lm(p_j)=m\cdot\lm(f_j)$ for some $m\in M$, and 
    \[p_{j+1}=p_j-m\cdot f_j.\]
\end{definition}
If $I$ is the ideal in $\F_2[\overline X]$ generated by $F$, note that $p\stackrel{F}\longrightarrow q$ implies $p-q\in I$. Therefore, if $p\stackrel{F}\longrightarrow0$, then $p\in I$. However, the opposite implication does not hold in general. A Gr\"obner basis for an ideal is its generating set for which this opposite implication does hold.
\begin{definition}
   Let $F\subset\F_2[\overline X]$ be a finite set of nonzero polynomials, and $I$ the ideal in $\F_2[\overline X]$ generated by $F$. We say that $F$ is a {\it Gr\"obner basis} for $I$ with respect to $\preceq$, if every $p\in I$ reduces to zero modulo $F$. 
\end{definition}
If $p\stackrel{F}\longrightarrow\overline p$, and if $\overline p$ cannot be further reduced modulo $F$, that is, if either $\overline p=0$ or $\lm(\overline p)$ is not divisible by any of the monomials from $\lm(F)$, we say that $\overline p$ is {\it a normal form of $p$ modulo $F$}. It is a theorem that if $F$ is a Gr\"obner basis (for the ideal generated by $F$), then every polynomial has a unique normal form, so we can speak of {\it the} normal form of a polynomial modulo $F$. So, we have the following theorem, which demonstrates the above mentioned main purpose of Gr\"obner bases.
\begin{theorem}\label{normalform}
    Let $F$ be a Gr\"obner basis for the ideal $I$ (with respect to a given monomial order). For an arbitrary polynomial $p\in\F_2[\overline X]$ denote by $\overline p$ its normal form modulo $F$. Then the following equivalence holds:
    \[p\in I \quad \Longleftrightarrow \quad \overline p=0.\]
\end{theorem}

An obvious consequence of this theorem is that if a nonzero polynomial $p$ cannot be reduced modulo $F$ (i.e., if $\overline p=p$), then $p\notin I$.

Note also that if $\overline p$ is the normal form of $p$ modulo $F$ (so that $\overline p$ cannot be reduced modulo $F$), then uniqueness of normal forms implies that none of the monomials from $M(\overline p)$ is divisible by any of the monomials from $\lm(F)$. Since $p-\overline p\in I$, this means that the cosets of monomials not divisible by any of the elements from $\lm(F)$ generate the quotient vector space $\F_2[\overline X]/I$. Additionally, the uniqueness of normal forms implies linear independence of these cosets. Therefore, the following theorem holds.
\begin{theorem}\label{additive basis}
    If $F$ is a Gr\"obner basis for the ideal $I\trianglelefteq\F_2[\overline X]$ (with respect to a given monomial order), then the set
    \[\Big\{m+I\,\,\Big|\,\, m\in M,\,\, (\forall f\in F)\, \lm(f)\nmid m\Big\}\]
    is a vector space basis for $\F_2[\overline X]/I$.
\end{theorem}

\subsection{Background on the cohomology algebra $H^*(\gr_{n,k})$}

We will be working with $\mathbb F_2$ coefficients exclusively, so in the rest of the paper we will write $H^*(X)$ for $H^*(X;\mathbb F_2)$, $\cl(X)$ for $\cl_{\mathbb F_2}(X)$, and $\zcl(X)$ for $\zcl_{\mathbb F_2}(X)$.

Let $k\ge2$ and $n\ge2k$ be integers, and let $\widetilde{\gamma}_{n,k}$ be the ($k$-dimensional) canonical vector bundle over the oriented Grassmann manifold $\gr_{n,k}$. Since this manifold is simply connected, the first Stiefel--Whitney class $w_1(\widetilde\gamma_{n,k})\in H^1(\gr_{n,k})$ is trivial. We will denote by $\w_i$ the Stiefel--Whitney classes $w_i(\widetilde\gamma_{n,k})\in H^i(\gr_{n,k})$ for $2\le i\le k$. Let $W_{n,k}$ be the subalgebra of $H^*(\gr_{n,k})$ generated by these classes $\w_i$, $2\le i\le k$. It is well known (see e.g.\ \cite[p.\ 197]{CP}) that $W_{n,k}$ is actually the image of the map $p^*: H^*(G_{n,k})\rightarrow H^*(\gr_{n,k})$ induced by the universal covering $p:\gr_{n,k}\arr G_{n,k}$ (where $G_{n,k}$ is the Grassmann manifold of {\it unoriented} $k$-subspaces in $\R^n$ and $p$ "forgets" the orientation), and that one has an isomorphism of graded algebras
\begin{equation}\label{eq:imp*2}
W_{n,k} \cong \frac{\mathbb F_2[w_2,\dots,w_k]}{I_{n,k}},
\end{equation}
where $I_{n,k}$ is the ideal generated by the polynomials $g_{n-k+1}^{(k)},\dots ,g_n^{(k)}$, which are defined by the relation
\begin{equation}\label{power_series}
(1+w_2+\cdots+w_k)(g_0^{(k)}+g_1^{(k)}+g_2^{(k)}+\cdots)=1.
\end{equation}

Via the isomorphism (\ref{eq:imp*2}) the Stiefel--Whitney class $\w_i\in W_{n,k}$ maps to the coset of the variable $w_i$, $2\le i\le k$ (it is understood that the degree of $w_i$ is $i$). Therefore, for an arbitrary polynomial $f=f(w_2,\ldots,w_k)\in\F_2[w_2,\ldots,w_k]$ we have the equivalence
\[f(w_2,\ldots,w_k)\in I_{n,k} \quad \Longleftrightarrow \quad f(\w_2,\ldots,\w_k)=0 \,\mbox{ in } \, W_{n,k}\subset H^*(\gr_{n,k}),\]
which will be used throughout the paper.

The following two relations are straightforward from (\ref{power_series}):
	\begin{equation}\label{eq_explicitly_g}
		g_r^{(k)} = \sum_{2a_2+\cdots +ka_k = r} \binom{a_2+\cdots +a_k}{a_2}\cdots\binom{a_{k-1} + a_k}{a_{k-1}}\,  w_2^{a_2}\cdots w_k^{a_k},\quad r\ge0
	\end{equation}
    (the sum is taken over all $(k-1)$-tuples $(a_2,\ldots,a_k)$ of nonnegative integers such that $2a_2+\cdots +ka_k = r$);
	\begin{equation}\label{eq_rec_g}
		g_r^{(k)} =   w_2 g_{r-2}^{(k)} + \cdots +   w_k g_{r-k}^{(k)}, \quad r\ge k.
	\end{equation}
    Additionally, if we define $g_r^{(k)}=0$ for each $r\in\{-1,-2,\ldots,-k+1\}$, it is easy to check that the last equation holds for all integers $r\ge1$. Furthermore, using induction on $s$, we get the following generalized version of relation (\ref{eq_rec_g}) (see e.g.\ \cite[(2.8)]{CJP}):
	\begin{equation}\label{eq_rec_g_gen}
		g_r^{(k)} \!=  w_2^{2^s} g_{r-2\cdot 2^s}^{(k)} + \cdots +   w_k^{2^s} g_{r-k2^s}^{(k)}=\!\sum_{j=2}^kw_j^{2^s} g_{r-j2^s}^{(k)}, \,\,\, r\geqslant1+k(2^s-1),\, s\geqslant0.
	\end{equation}

Another consequence of (\ref{eq_rec_g}) is the fact that not only $g_{n-k+1}^{(k)},\dots ,g_n^{(k)}$ but all polynomials $g_r^{(k)}$ with $r\ge n-k+1$, belong to $I_{n,k}$:
\begin{equation}\label{g_rinIn,k}
    g_r^{(k)}\in I_{n,k} \quad \mbox{ for all } r\ge n-k+1.
\end{equation}
In particular, the polynomials that generate $I_{n+1,k}=(g_{n-k+2}^{(k)},\dots ,g_{n+1}^{(k)})$ belong to $I_{n,k}$, and therefore, the sequence of ideals $I_{n,k}$ decreases with $n$:
    \begin{equation}\label{In,k opada}
       I_{n,k}\supseteq I_{n+1,k} \quad \mbox{ for all } k\ge2 \mbox{ and all } n\ge2k. 
    \end{equation}

A useful fact concerning the polynomials $g_r^{(k)}$ is proved in \cite[Lemma 2.3(ii)]{Korbas}: for $k\in\{3,4\}$ and all $t\ge2$
\begin{equation}\label{eq:g=0}
g_{2^t-3}^{(k)} = 0.
\end{equation}
This equation, along with (\ref{eq_rec_g}), implies 
$g_{2^t+1}^{(4)}=w_2 g_{2^t-1}^{(4)} + w_3 g_{2^t-2}^{(4)}$,
and consequently
\begin{equation}\label{I_2^t=I_2^t+1}
 I_{2^t+1,4}=\big(g_{2^t-2}^{(4)},g_{2^t-1}^{(4)},g_{2^t}^{(4)},g_{2^t+1}^{(4)}\big)=\big(g_{2^t-2}^{(4)},g_{2^t-1}^{(4)},g_{2^t}^{(4)}\big)=I_{2^t,4}.   
\end{equation}
From (\ref{eq:imp*2}) we now obtain the isomorphism of graded algebras
\begin{equation}\label{W_2^t=W_2^t+1}
 W_{2^t,4}\cong W_{2^t+1,4},   
\end{equation}
via which the Stiefel--Whitney class $\w_i\in H^i(\gr_{2^t,4})$ corresponds to the Stiefel--Whitney class $\w_i\in H^i(\gr_{2^t+1,4})$, $i=2,3,4$.

Similarly, using (\ref{eq_rec_g}) and (\ref{eq:g=0}) one can check that
\begin{equation}\label{I_2^t=I_2^t-1}
 I_{2^t,3}=\big(g_{2^t-2}^{(3)},g_{2^t-1}^{(3)},g_{2^t}^{(3)}\big)=\big(g_{2^t-2}^{(3)},g_{2^t-1}^{(3)}\big)=I_{2^t-1,3}  
\end{equation}
(cf.\ \cite[Remark 3.16]{CP-cl}).

Another useful fact was proved in \cite[Lemma 4.6]{CJP}: the variable $w_4$ does not appear in $g_{2^t-3+2^i}^{(4)}$ for any $i\in\{0,1,\ldots,t-1\}$. On the other hand, it is obvious from (\ref{eq_explicitly_g}) that the polynomial $g_r^{(3)}$ is the reduction of $g_r^{(4)}$ modulo $w_4$. We conclude
\begin{equation}\label{3=4}
g_{2^t-3+2^i}^{(4)}=g_{2^t-3+2^i}^{(3)} \quad \mbox{ for all } i\in\{0,1,\ldots,t-1\}.
    \end{equation}

In general, the ideal $I_{n,k}\trianglelefteq\mathbb F_2[w_2,\ldots,w_k]\subset\mathbb F_2[w_2,\ldots,w_k,w_{k+1}]$ need not be a subset of $I_{n,k+1}\trianglelefteq\mathbb F_2[w_2,\ldots,w_k,w_{k+1}]$. Nevertheless, this does hold for $k=3$ and $n$ a power of two. More precisely, we have the following lemma.

\begin{lemma}\label{utapanje}
    For all $t\ge3$ one has $I_{2^t,3}=I_{2^t,4}\cap\mathbb F_2[w_2,w_3]$.
\end{lemma}
\begin{proof}
    We already know (see (\ref{I_2^t=I_2^t+1}) and (\ref{I_2^t=I_2^t-1})) that
    \[I_{2^t,3} = \big(g_{2^t-2}^{(3)},g_{2^t-1}^{(3)}\big)\trianglelefteq\mathbb F_2[w_2,w_3] \mbox{ and } I_{2^t,4} = \big(g_{2^t-2}^{(4)}, g_{2^t-1}^{(4)}, g_{2^t}^{(4)}\big)\trianglelefteq\mathbb F_2[w_2,w_3,w_4].\]
    
    We certainly have $I_{2^t,3}\subseteq\mathbb F_2[w_2,w_3]$, and so, in order to establish the containment $I_{2^t,3}\subseteq I_{2^t,4}\cap\mathbb F_2[w_2,w_3]$, it suffices to check that $g_{2^t-2}^{(3)},g_{2^t-1}^{(3)}\in I_{2^t,4}$ (since $I_{2^t,4}$ is an ideal in $\mathbb F_2[w_2,w_3,w_4]$). According to (\ref{3=4}), we have $g_{2^t-2}^{(3)} = g_{2^t-2}^{(4)}\in I_{2^t,4}$ and $g_{2^t-1}^{(3)} = g_{2^t-1}^{(4)}\in I_{2^t,4}$. 

    \medskip

    For the reverse containment, let $f\in I_{2^t,4}$ be a polynomial which does not contain the variable $w_4$ (i.e., $f\in\mathbb F_2[w_2,w_3]$). Write $f$ in the form \[f=p_2g_{2^t-2}^{(4)}+p_1g_{2^t-1}^{(4)}+p_0g_{2^t}^{(4)},\]
    for some $p_0,p_1,p_2\in\mathbb F_2[w_2,w_3,w_4]$. If we reduce this equation modulo $w_4$ (this reduction $\mathbb F_2[w_2,w_3,w_4]\rightarrow\mathbb F_2[w_2,w_3,w_4]/(w_4)=\mathbb F_2[w_2,w_3]$ is a ring epimorphism), using the fact $f\in\mathbb F_2[w_2,w_3]$ and (\ref{I_2^t=I_2^t-1}) we obtain
    \[f=\overline p_2g_{2^t-2}^{(3)}+\overline p_1g_{2^t-1}^{(3)}+\overline p_0g_{2^t}^{(3)}\in I_{2^t,3}\]
    ($\overline p_0,\overline p_1,\overline p_2\in\mathbb F_2[w_2,w_3]$ are reductions of $p_0,p_1,p_2$ respectively).
\end{proof}

From now on we will be working mostly with the case $k=4$, and for that reason we will abbreviate $g_r^{(4)}$ to $g_r$. 

A Gr\"obner basis for the ideal $I_{2^t,3}=I_{2^t-1,3}$, with respect to the lexicographic order in which $w_2>w_3$, was detected in \cite{Fukaya}. It is the set
\[F=\big\{g_{2^t-3+2^i}\mid 0\le i\le t-1\big\}=\big\{g_{2^t-3+2^i}^{(3)}\mid 0\le i\le t-1\big\}\]
(see (\ref{3=4})), and the set of the leading monomials is
\begin{equation}\label{LM(F)}
   \lm(F)=\big\{w_2^{2^{t-1}-2^i}w_3^{2^i-1}\mid 0\le i\le t-1\big\}
\end{equation}
(see \cite[Proposition 3.4]{CP}).
Gr\"obner bases $F_n$ for the ideals $I_{n,4}$ in the cases $n\in\{2^t-2,2^t-1,2^t,2^t+1\}$ ($t\ge4$), with respect to the lexicographic order in which $w_4>w_2>w_3$, were presented in \cite[Lemma 3.4]{Rusin}, and they are as follows:
\[F_{2^t}=F_{2^t+1}= F \cup \{g_{2^t } \}, \quad F_{2^t-1}=F \cup  \{g_{2^t -4} \}, \quad F_{2^t-2}=F \cup  \{g_{2^t -4}, g_{2^t -5} \}.\]
In fact, the claim that $F_{2^t}$ is a Gr\"obner basis for $I_{2^t,4}=I_{2^t+1,4}$ is easily checked to be true for $t=3$ as well. The sets of the leading monomials are:
\begin{align*}
    \lm(F_{2^t})&=\lm(F)\cup\big\{w_4^{2^{t-2}}\big\},\\
    \lm(F_{2^t-1})&=\lm(F)\cup\big\{w_4^{2^{t-2}-1}\big\},\\
    \lm(F_{2^t-2})&=\lm(F)\cup\big\{w_4^{2^{t-2}-1},w_3w_4^{2^{t-2}-2}\big\}.
\end{align*}

We now present two lemmas which, along with Lemma \ref{utapanje}, establish some crucial connections between the ideal $I_{2^t,3}$ in $\F_2[w_2,w_3]$ on one side, and the ideals $I_{2^t,4}$ and $I_{2^t-1,4}$ in $\F_2[w_2,w_3,w_4]$ on the other.

\begin{lemma}\label{lem: Vuk}
    If $p_j\in\mathbb F_2[w_2,w_3]$, $0\le j\le2^{t-2}-1$, $t\ge 3$,  then
    \[\sum_{j=0}^{2^{t-2}-1}p_jw_4^j\in I_{2^t,4} \quad \Longleftrightarrow \quad p_j\in I_{2^t,3} \mbox{ for all } j\in\{0,1,\ldots,2^{t-2}-1\}.\]
\end{lemma}
\begin{proof} 
Suppose $\sum_{j=0}^{2^{t-2}-1}p_jw_4^j\in I_{2^t,4}$. For $j\in\{0,1,\ldots,2^{t-2}-1\}$ let $\overline p_j\in\mathbb F_2[w_2,w_3]$ be the normal form of $p_j$ modulo $F$ (which is a Gr\"obner basis for $I_{2^t,3}$). Then $p_j\in I_{2^t,3}$ if and only if $\overline p_j=0$ (Theorem \ref{normalform}), and since $F\subset I_{2^t,4}$ we have $\sum_{j=0}^{2^{t-2}-1}\overline p_jw_4^j\in I_{2^t,4}$. But this sum is a polynomial in which no monomial is divisible by any of the leading monomials $\lm(F_{2^t})=\lm(F)\cup\{w_4^{2^{t-2}}\}$ from Gr\"obner basis $F_{2^t}$ of $I_{2^t,4}$. This is possible only if $\sum_{j=0}^{2^{t-2}-1}\overline p_jw_4^j=0$, i.e., $\overline p_j=0$ for all $j$.

Conversely, if $p_j\in I_{2^t,3}$ for all $j\in\{0,1,\ldots,2^{t-2}-1\}$, then by Lemma \ref{utapanje}, $p_j\in I_{2^t,4}$ for all $j\in\{0,1,\ldots,2^{t-2}-1\}$, and so $\sum_{j=0}^{2^{t-2}-1}p_jw_4^j\in I_{2^t,4}$.
\end{proof}
When it comes to the ideal $I_{2^t-1,4}$, using its Gr\"obner basis $F_{2^t-1}$ (whose set of leading monomials is $\lm(F)\cup\{w_4^{2^{t-2}-1}\}$) and the fact $I_{2^t,4}\subseteq I_{2^t-1,4}$, one can easily adjust the proof of Lemma \ref{lem: Vuk} to get a proof of the following lemma.
\begin{lemma}\label{lem: Vuk2}
    If $p_j\in\mathbb F_2[w_2,w_3]$, $0\le j\le2^{t-2}-2$, $t\ge4$, then
    \[\sum_{j=0}^{2^{t-2}-2}p_jw_4^j\in I_{2^{t}-1,4} \quad \Longleftrightarrow \quad p_j\in I_{2^t,3} \mbox{ for all } j\in\{0,1,\ldots,2^{t-2}-2\}.\]
\end{lemma}

\section{Heights of the Stiefel--Whitney classes}

In this section we determine heights of the Stiefel--Whitney classes $\w_2, \w_3$ and $\w_4$ in $H^*(\gr_{n,4})$ for $n\in\{2^t-2,2^t-1,2^t,2^t+1\}$.

\label{sec:heights}
\label{sec:3}

\begin{theorem}\label{prop:ht}
    Let $n\in\{2^t-2,2^t-1,2^t,2^t+1\}$, $t\ge4$. In $H^*(\gr_{n,4})$ we have
    \begin{enumerate}
        \item $\htt(\w_2) = 2^t-4$,
        \item $\htt(\w_3) = 2^{t-1}-2$.
    \end{enumerate}
    \end{theorem}
    \begin{proof}
    We know that $\w_2^b\w_3^c=0$ in $H^*(\gr_{n,4})$ if and only if $w_2^bw_3^c\in I_{n,4}$. Since $I_{2^t+1,4}=I_{2^t,4}$ (see (\ref{I_2^t=I_2^t+1})), from
    Lemmas \ref{lem: Vuk} and \ref{lem: Vuk2} (applied to $p_0=w_2^bw_3^c$ and $p_j=0$ for $j>0$) we conclude that for $n\in\{2^t-1,2^t,2^t+1\}$, $w_2^bw_3^c\in I_{n,4}$ is equivalent to $w_2^bw_3^c\in I_{2^t,3}$. Therefore, the heights of $\w_2$ and $\w_3$ in these cases are equal to the heights of $\w_2$ and $\w_3$ in $H^*(\gr_{2^t,3})$.
    These heights are calculated in \cite[Theorems 1.2 and 1.3]{CP-cl}, and they are as stated.

    \medskip

    For the remaining case $n=2^t-2$ we have $w_2^{2^t-3},w_3^{2^{t-1}-1}\in I_{2^t-1,4}\subseteq I_{2^t-2,4}$ (see (\ref{In,k opada})), and it remains to check that $w_2^{2^t-4},w_3^{2^{t-1}-2}\notin I_{2^t-2,4}$. We prove this using the Gr\"obner basis $F_{2^t-2}$ for $I_{2^t-2,4}$. No monomial from $\lm(F_{2^t-2})=\lm(F)\cup\big\{w_4^{2^{t-2}-1},w_3w_4^{2^{t-2}-2}\big\}$ divides $w_3^{2^{t-1}-2}$, i.e., $w_3^{2^{t-1}-2}$ cannot be reduced modulo $F_{2^t-2}$, and so $w_3^{2^{t-1}-2}\notin I_{2^t-2,4}$. 

    From \cite[(4.3)]{CP-cl} we have $w_2^{2^t-4}+w_2^{2^{t-2}-1}w_3^{2^{t-1}-2}\in I_{2^t-1,3}=I_{2^t,3}$, and we know that $I_{2^t,3}\subseteq I_{2^t,4}\subseteq I_{2^t-2,4}$ (Lemma \ref{utapanje} and (\ref{In,k opada})). So $w_2^{2^t-4}+w_2^{2^{t-2}-1}w_3^{2^{t-1}-2}\in I_{2^t-2,4}$, which means that it is enough to prove $w_2^{2^{t-2}-1}w_3^{2^{t-1}-2}\notin I_{2^t-2,4}$. It is shown in \cite[p.\ 277]{CP-cl} that no monomial from $\lm(F)$ divides $w_2^{2^{t-2}-1}w_3^{2^{t-1}-2}$, and obviously, none of $w_4^{2^{t-2}-1}$ and $w_3w_4^{2^{t-2}-2}$ divides it either. Since $F_{2^t-2}$ is a Gr\"obner basis, this implies $w_2^{2^{t-2}-1}w_3^{2^{t-1}-2}\notin I_{2^t-2,4}$, and consequently $w_2^{2^t-4}\notin I_{2^t-2,4}$.
    \end{proof}
\begin{remark}
    This theorem (and its proof) is valid for $t=3$ and $n \in \{2^t,2^t+1\}$, i.e., for $n\in\{8,9\}$, as well.
\end{remark}

We are left to determine the height of $\w_4$ for $n$ as in the previous theorem. Let us first formulate the assertion.

\begin{theorem}\label{htw_4}
    For $t\ge4$ and $j\in\{0,1,2\}$, in $H^*(\gr_{2^t-j,4})$ one has
    \[\htt(\w_4) = 2^{t-1}-1-j.\]
    Also, for $t\ge3$, in $H^*(\gr_{2^t+1,4})$ one has $\htt(\w_4) = 2^{t-1}-1$.
    \end{theorem}

The rest of this section is devoted to proving Theorem \ref{htw_4}. We will separate the proof into three parts: $j=0$, $j=1$ and $j=2$. These parts are established in Propositions \ref{prop:ht(w4)}, \ref{prop: ht(w_4)_2^t-1} and \ref{prop:ht(w4)2^t-2} respectively. Due to (\ref{W_2^t=W_2^t+1}), the height of $\w_4$ in $H^*(\gr_{2^t+1,4})$ is the same as in $H^*(\gr_{2^t,4})$, and so the part $j=0$ (Proposition \ref{prop:ht(w4)}) includes this case as well. 

Let us first recall an equation from \cite[Lemma 3.2]{Rusin}: for $t\ge3$
\begin{equation}\label{eq: Rusin}
  g_{2^t-4} = \sum_{i= 0}^{t-2}w_4^{2^i-1}\big(g_{2^{t-i}-4}^{(3)}\big)^{2^i}.
\end{equation}
Note, however, that this equation holds for $t=2$ as well (in that case (\ref{eq: Rusin}) reduces to $1=1$). 

Another fact that will be used in the upcoming computations was proved in \cite[(4.7)]{CJP}:
\begin{equation}\label{w_3kvadrat}
    w_3g_r^2=g_{2r+3} \quad \mbox{ for all } r\ge-3.
\end{equation}

The next lemma expresses $w_4^{2^{t-1}-1}$ in terms of polynomials $g_r$, and we will use it for the proof of Theorem \ref{htw_4} for both $j=0$ and $j=1$.

\begin{lemma} \label{lem:rel-2^(t+1)}
For all integers $t\ge3$ one has
    \[w_4^{2^{t-1}-1}=g_{2^{t+1}-4}+w_2^{2^{t-1}}g_{2^t-4} + w_3^{2^{t-1}}g_{2^{t-1}-4}.\]
\end{lemma}
\begin{proof}
    We use equation (\ref{eq: Rusin}) to calculate:
    \begin{align*}
       g_{2^{t+1}-4}&+w_2^{2^{t-1}}g_{2^t-4} + w_3^{2^{t-1}}g_{2^{t-1}-4} \\
       =&\sum_{i= 0}^{t-1}w_4^{2^i-1}\big(g_{2^{t+1-i}-4}^{(3)}\big)^{2^i}+w_2^{2^{t-1}}\!\sum_{i= 0}^{t-2}w_4^{2^i-1}\big(g_{2^{t-i}-4}^{(3)}\big)^{2^i}\\&+w_3^{2^{t-1}}\sum_{i= 0}^{t-3}w_4^{2^i-1}\big(g_{2^{t-1-i}-4}^{(3)}\big)^{2^i} \\
       =&\,w_4^{2^{t-1}-1}+\sum_{i= 0}^{t-2}w_4^{2^i-1}\big(g_{2^{t+1-i}-4}^{(3)}\big)^{2^i}+\sum_{i= 0}^{t-2}w_4^{2^i-1}\big(w_2^{2^{t-1-i}}g_{2^{t-i}-4}^{(3)}\big)^{2^i}\\
       &+\sum_{i= 0}^{t-2}w_4^{2^i-1}\big(w_3^{2^{t-1-i}}g_{2^{t-1-i}-4}^{(3)}\big)^{2^i} \\
       =&\,w_4^{2^{t-1}-1}+\sum_{i= 0}^{t-2}w_4^{2^i-1}\bigg(g_{2^{t+1-i}-4}^{(3)}+w_2^{2^{t-1-i}}g_{2^{t-i}-4}^{(3)}+w_3^{2^{t-1-i}}g_{2^{t-1-i}-4}^{(3)}\bigg)^{2^i}
    \end{align*}
    The sums in the third and the fifth row are equal, because the summand for $i=t-2$ vanishes (since $g_{-2}^{(3)}=0$). Finally, by (\ref{eq_rec_g_gen}) (applied to $k=3$, $r=2^{t+1-i}-4$ and $s=t-1-i$) the expression in the last brackets is zero, and we are done.
\end{proof}

\begin{proposition}\label{prop:ht(w4)}
    Let $t\ge3$ and $n\in\{2^t,2^t+1\}$. Then in $H^*(\gr_{n,4})$ we have 
    \[\htt(\w_4) = 2^{t-1}-1.\]
\end{proposition}
\begin{proof}
We know that $I_{2^t,4}=I_{2^t+1,4}$ (see (\ref{I_2^t=I_2^t+1})), and so it is enough to prove that $w_4^{2^{t-1}}\in I_{2^t,4}$ and $w_4^{2^{t-1}-1}\notin I_{2^t,4}$.

The equation (\ref{eq_rec_g_gen}) for $k=4$, $r=2^{t+1}$ and $s=t-1$ is
\begin{equation}\label{eq:ht-w4-1}
g_{2^{t+1}} = w_2^{2^{t-1}}g_{2^t} + w_3^{2^{t-1}}g_{2^{t-1}} + w_4^{2^{t-1}}.\end{equation}
Recall that $g_r\in I_{n,4}$ for all $r\ge n-3$ (see (\ref{g_rinIn,k})), and therefore
 $g_{2^{t+1}}, g_{2^t}\in I_{2^t,4}$, while $w_3^{2^{t-1}}\in I_{2^t,4}$ since $\htt(\w_3) = 2^{t-1}-2$ (Theorem \ref{prop:ht}). Equation (\ref{eq:ht-w4-1}) now implies $w_4^{2^{t-1}}\in I_{2^t,4}$.

Using again the fact $w_3^{2^{t-1}}\in I_{2^t,4}$, as well as $g_{2^{t+1}-4}\in I_{2^t,4}$, from Lemma \ref{lem:rel-2^(t+1)} we conclude that $w_4^{2^{t-1}-1}\notin I_{2^t,4}$ if and only if $w_2^{2^{t-1}}g_{2^t-4}\notin I_{2^t,4}$. Since $\lm(g_{2^t-4}) = w_4^{2^{t-2}-1}$, by Lemma ~\ref{lem: Vuk} it suffices to prove $w_2^{2^{t-1}}\notin I_{2^t,3}$, which is true because $\htt(\w_2) = 2^t-4\ge2^{t-1}$ in $H^*(\gr_{2^t,3})$ (see \cite[Theorem 1.2]{CP-cl}). 
\end{proof}

So, we have verified Theorem \ref{htw_4} for $j=0$ (including the case $n=2^t+1$). We now move on to the case $j=1$. 

\begin{proposition}\label{prop: ht(w_4)_2^t-1}
    Let $t\ge4$. For $\w_4 \in H^*(\gr_{2^t-1,4})$ we have
    \[\htt(\w_4) = 2^{t-1}-2.\]
\end{proposition}
\begin{proof}
  We need to check that $w_4^{2^{t-1}-1}\in I_{2^t-1,4}$ and $w_4^{2^{t-1}-2}\notin I_{2^t-1,4}$. The first assertion follows from Lemma \ref{lem:rel-2^(t+1)}, since $g_{2^{t+1}-4},g_{2^t-4},w_3^{2^{t-1}}\in I_{2^t-1,4}$ (by (\ref{g_rinIn,k}) and Theorem \ref{prop:ht}).

  For the second one, we use (\ref{eq: Rusin}):
  \[\big(g_{2^t-4}\big)^2=\sum_{i= 0}^{t-2}w_4^{2^{i+1}-2}\big(g_{2^{t-i}-4}^{(3)}\big)^{2^{i+1}}=w_4^{2^{t-1}-2}+\sum_{i= 0}^{t-3}w_4^{2^{i+1}-2}\big(g_{2^{t-i}-4}^{(3)}\big)^{2^{i+1}}.\]
  Using the fact $g_{2^t-4}\in I_{2^t-1,4}$ again, we conclude that the claim $w_4^{2^{t-1}-2}\notin I_{2^t-1,4}$ is equivalent to the assertion that the last sum does not belong to $I_{2^t-1,4}$. By Lemma \ref{lem: Vuk2} it suffices to find $i\in\{0,1,\ldots,t-3\}$ with the property $\big(g_{2^{t-i}-4}^{(3)}\big)^{2^{i+1}}\notin I_{2^t,3}$. For $i=t-3$ this polynomial is $\big(g_4^{(3)}\big)^{2^{t-2}}=w_2^{2^{t-1}}$ (since $g_4^{(3)}=w_2^2$), and the fact $w_2^{2^{t-1}}\notin I_{2^t,3}$ was established in the proof of the previous proposition.
\end{proof}
  
To complete the proof of Theorem \ref{htw_4} we are left to prove it for $j=2$. In that case Theorem \ref{htw_4} comes down to $w_4^{2^{t-1}-2}\in I_{2^t-2,4}$ and $w_4^{2^{t-1}-3}\notin I_{2^t-2,4}$. In the previous proof, in order to show $w_4^{2^{t-1}-1}\in I_{2^t-1,4}$, we used a convenient representation of $w_4^{2^{t-1}-1}$, given in Lemma \ref{lem:rel-2^(t+1)}. It would be helpful to obtain something similar for $w_4^{2^{t-1}-2}$. 

\begin{lemma}\label{lem: uzas1}
    Let $t \ge 4$. In $\F_2[w_2,w_3,w_4]$ we have the following identity:
    \[w_4^{2^{t-1}-2} = \alpha_tg_{2^t-2}+ (g_{2^t-4})^2 + \beta_tg_{2^t-5}  + w_3^{2^{t-1}-1}  g_{2^{t-1} - 5},\] 
    for some polynomials $\alpha_t,\beta_t\in\F_2[w_2,w_3,w_4]$.
\end{lemma}
\begin{proof}
If we multiply the equation from Lemma \ref{lem:rel-2^(t+1)} by $w_3$, and if we apply (\ref{eq_rec_g}) (for $k=4$ and $r = 2^{t+1}-1, 2^t-1, 2^{t-1} -1$), we obtain
\begin{align*}
  w_3w_4^{2^{t-1}-1}=& \,w_3g_{2^{t+1}-4}+w_2^{2^{t-1}}w_3g_{2^t-4} + w_3^{2^{t-1}}w_3g_{2^{t-1}-4}\\
  =&\,w_4g_{2^{t+1}-5}+g_{2^{t+1}-1}+w_2^{2^{t-1}}(w_4g_{2^t-5}+g_{2^t-1}) \\
  &+ w_3^{2^{t-1}}(w_4g_{2^{t-1}-5}+g_{2^{t-1}-1})\\
  =&\,w_4\big(g_{2^{t+1}-5}+w_2^{2^{t-1}}g_{2^t-5} + w_3^{2^{t-1}}g_{2^{t-1}-5}\big)\\
  &+g_{2^{t+1}-1}+w_2^{2^{t-1}}g_{2^t-1} + w_3^{2^{t-1}}g_{2^{t-1}-1}.
  \end{align*}
  According to (\ref{3=4}), 
  \[g_{2^{t+1}-1}+w_2^{2^{t-1}}g_{2^t-1} + w_3^{2^{t-1}}g_{2^{t-1}-1}=g_{2^{t+1}-1}^{(3)}+w_2^{2^{t-1}}g_{2^t-1}^{(3)} + w_3^{2^{t-1}}g_{2^{t-1}-1}^{(3)},\] 
  and this vanishes by (\ref{eq_rec_g_gen}) (applied to $k=3$, $r=2^{t+1}-1$ and $s=t-1$). Therefore
  \[w_3w_4^{2^{t-1}-1}=w_4\big(g_{2^{t+1}-5}+w_2^{2^{t-1}}g_{2^t-5} + w_3^{2^{t-1}}g_{2^{t-1}-5}\big),\]
  and canceling out $w_4$ we get
  \begin{equation}\label{eqq}
     w_3w_4^{2^{t-1}-2}=g_{2^{t+1}-5}+w_2^{2^{t-1}}g_{2^t-5} + w_3^{2^{t-1}}g_{2^{t-1}-5}. 
  \end{equation}
Now, to obtain the desired equation, we wish to cancel out $w_3$, and for that purpose we write each of the summands on the right-hand side as a convenient multiple of $w_3$.
By (\ref{w_3kvadrat}), $g_{2^{t+1}-5}=w_3(g_{2^t-4})^2$ and $g_{2^t-5}=w_3(g_{2^{t-1}-4})^2$. The monomial $w_2^{2^{t-1}}$ is the leading monomial of $g_{2^t}^{(3)}$, and all other monomials of $g_{2^t}^{(3)}$ are divisible by $w_3$ (this is straightforward from (\ref{eq_explicitly_g})). Therefore, $w_2^{2^{t-1}}+g_{2^t}^{(3)}=w_3\beta_t$ for some polynomial $\beta_t\in\F_2[w_2,w_3]\subset\F_2[w_2,w_3,w_4]$. On the other hand, $g_{2^t}^{(3)}=w_2g_{2^t-2}^{(3)}=w_2g_{2^t-2}$ ((\ref{eq_rec_g}), (\ref{eq:g=0}) and (\ref{3=4})). Therefore
\[w_2^{2^{t-1}}g_{2^t-5}=\big(w_2^{2^{t-1}}+g_{2^t}^{(3)}+w_2g_{2^t-2}\big)g_{2^t-5}=w_3\beta_tg_{2^t-5}+w_2g_{2^t-2}g_{2^t-5}.\]
By collecting all these facts, from (\ref{eqq}) we obtain:
\[w_3w_4^{2^{t-1}-2}=w_3(g_{2^t-4})^2+w_3\beta_tg_{2^t-5}+w_2g_{2^t-2}w_3(g_{2^{t-1}-4})^2 + w_3^{2^{t-1}}g_{2^{t-1}-5}.\]
Finally, canceling out $w_3$ and defining $\alpha_t:=w_2(g_{2^{t-1}-4})^2$ finishes the proof.
\end{proof}

\begin{proposition}\label{prop:ht(w4)2^t-2}
     Let $t \ge 4$. For $\w_4 \in H^*(\gr_{2^t-2,4})$ we have
    \[\htt(\w_4) = 2^{t-1}-3.\]
\end{proposition}

\begin{proof}
Since $g_r\in I_{2^t-2,4}$ for all $r\ge2^t-5$ (by (\ref{g_rinIn,k})) and $w_3^{2^{t-1}-1}\in I_{2^t-2,4}$ (because $\htt(\w_3)=2^{t-1}-2$, by Theorem \ref{prop:ht}), Lemma \ref{lem: uzas1} proves that $w_4^{2^{t-1}-2}\in I_{2^t-2,4}$. We are left to check that $w_4^{2^{t-1}-3}\notin I_{2^t-2,4}$.

By (\ref{eq: Rusin}) we have
\begin{align*}
  g_{2^t-4}(g_{2^{t-1}-4})^2 = &\sum_{i= 0}^{t-2}w_4^{2^i-1}\big(g_{2^{t-i}-4} ^{(3)}\big)^{2^i}\cdot \sum_{i= 0}^{t-3}w_4^{2^{i+1}-2}\big(g_{2^{t-1-i}-4} ^{(3)}\big)^{2^{i+1}}\\
  =&\sum_{i= 0}^{t-2}w_4^{2^i-1}\big(g_{2^{t-i}-4} ^{(3)}\big)^{2^i}\cdot \sum_{i= 1}^{t-2}w_4^{2^i-2}\big(g_{2^{t-i}-4}^{(3)}\big)^{2^i}\\
  =&\,\bigg(w_4^{2^{t-2}-1}+g_{2^t-4}^{(3)}+\sum_{i= 1}^{t-3}w_4^{2^i-1}\big(g_{2^{t-i}-4}^{(3)}\big)^{2^i}\bigg)\\
  &\cdot\bigg(w_4^{2^{t-2}-2}+\sum_{i= 1}^{t-3}w_4^{2^i-2}\big(g_{2^{t-i}-4}^{(3)}\big)^{2^i}\bigg)
\end{align*}
If $S$ denotes the sum $\sum_{i= 1}^{t-3}w_4^{2^i-2}\big(g_{2^{t-i}-4}^{(3)}\big)^{2^i}$, then
\begin{align*}
  g_{2^t-4}(g_{2^{t-1}-4})^2 = &\,\big(w_4^{2^{t-2}-1}+g_{2^t-4}^{(3)}+w_4S\big)\big(w_4^{2^{t-2}-2}+S\big)\\
  = &\,w_4^{2^{t-1}-3}+w_4^{2^{t-2}-2}g_{2^t-4}^{(3)}+g_{2^t-4}^{(3)}S+w_4S^2.
\end{align*}
Let $p=w_4^{2^{t-2}-2}g_{2^t-4}^{(3)}+g_{2^t-4}^{(3)}S+w_4S^2$, so that $g_{2^t-4}(g_{2^{t-1}-4})^2 =w_4^{2^{t-1}-3}+p$. Since $g_{2^t-4}\in I_{2^t-2,4}$, one has $w_4^{2^{t-1}-3}\notin I_{2^t-2,4}$ if and only if $p\notin I_{2^t-2,4}$. The exponent of $w_4$ in a monomial from $g_{2^t-4}^{(3)}S+w_4S^2$ is at most $2^{t-2}-3$, and so 
\[\lm(p)=w_4^{2^{t-2}-2}\cdot\lm\big(g_{2^t-4}^{(3)}\big)=w_4^{2^{t-2}-2}w_2^{2^{t-1}-2}.\]
However, for the Gr\"obner basis $F_{2^t-2}$ of $I_{2^t-2,4}$ we know that
\[\lm(F_{2^t-2})=\big\{w_2^{2^{t-1}-2^i}w_3^{2^i-1}\mid 0\le i\le t-1\big\}\cup\big\{w_4^{2^{t-2}-1},w_3w_4^{2^{t-2}-2}\big\}\]
(see (\ref{LM(F)})ff), so $\lm(p)$ is not divisible by any of the leading monomials from $F_{2^t-2}$, i.e., $p$ cannot be reduced modulo $F_{2^t-2}$. We conclude $p\notin I_{2^t-2,4}$, and the proof is complete.
\end{proof}

\section{Cup-length of $\gr_{2^t+1,4}$}
\label{sec:cl}
\label{sec:4}

The main result of \cite{Rusin} states that for all integers $t\ge4$ and $j\in\{0,1,2\}$ one has 
\[\cl(\gr_{2^t-j,4})=2^t+2^{t-2}-4-j.\]
(It is not hard to see that the proof of the result for $j=0$ actually goes through for $t=3$ as well, so that $\cl(\gr_{8,4})=6$.)

We are going to extend this result and prove that $\cl(\gr_{2^t+1,4})=2^t+2^{t-2}-4$ for all $t\ge3$. Actually, a part of the proof is already contained in \cite{Rusin}. Namely, by definition all elements of $W_{2^t,4}$ are polynomials in $\w_2$, $\w_3$ and $\w_4$, and so $\cl(W_{2^t,4})$ is reached by some monomial $\w_2^b\w_3^c\w_4^d$. Therefore, Lemmas 4.1 and 4.4 from \cite{Rusin} imply $\cl(W_{2^t,4})=2^t+2^{t-2}-5$, and since
 $W_{2^t+1,4}\cong W_{2^t,4}$ (see (\ref{W_2^t=W_2^t+1})), we have
 \begin{equation}\label{clW}
     \cl(W_{2^t+1,4})=2^t+2^{t-2}-5.
 \end{equation}
 
 In transition from the subalgebra $W_{2^t+1,4}$ to the algebra $H^*(\gr_{2^t+1,4})$, a useful result is obtained in \cite[Theorem 6.6]{bane-marko}. In the case we are interested in (by the notation used in \cite{bane-marko}, it is the case $n+4=2^t+1$) it states that the so-called characteristic rank of the vector bundle $\widetilde\gamma_{2^t+1,4}$ is equal to $2^t-1$. The characteristic rank of $\widetilde\gamma_{2^t+1,4}$ is defined as the maximal integer $q$ such that $H^j(\gr_{2^t+1,4})\subseteq W_{2^t+1,4}$ for all $j\in\{0,1,\ldots,q\}$. We conclude that for a homogeneous class $y\in H^*(\gr_{2^t+1,4})$ the following implication holds
 \begin{equation}\label{charrank}
     y\notin W_{2^t+1,4} \quad \Longrightarrow \quad |y|\ge2^t.
 \end{equation}

We are now ready to prove the announced result.
\begin{theorem}\label{cor:cl(G2^t+1,4)}
    Let $t\ges 3$. Then
    \[\cl(\gr_{2^t+1,4}) = 2^t+2^{t-2} - 4.\]
\end{theorem}

\begin{proof}
We know that $\htt(\w_2) = 2^t-4$ in $H^*(\gr_{2^t,3})$ (see \cite[Theorem 1.2]{CP-cl}), which means that $w_2^{2^t-4}\notin I_{2^t,3}$. By Lemma \ref{lem: Vuk}, $w_2^{2^t-4}w_4^{2^{t-2}-1}\notin I_{2^t,4}=I_{2^t+1,4}$, that is, $\w_2^{2^t-4}\w_4^{2^{t-2}-1}\neq0$ in $H^*(\gr_{2^t+1,4})$. The degree of this class is $2^{t+1}-8+2^t-4=3\cdot2^t-12$, while the dimension of the manifold $\gr_{2^t+1,4}$ is $4\cdot2^t-12$. By Poincar\'e duality, there is a class $y\in H^{2^t}(\gr_{2^t+1,4})$ with $\w_2^{2^t-4}\w_4^{2^{t-2}-1}y\neq0$. We conclude that $\cl(\gr_{2^t+1,4})\ge2^t+2^{t-2} - 4$.

We are left to prove $\cl(\gr_{2^t+1,4})\le2^t+2^{t-2} - 4$. Suppose to the contrary that there exists a nonzero product $x_1\cdots x_sy_1\cdots y_m$ of classes of positive degree such that $s+m\ge2^t+2^{t-2}-3$, where $x_1,\ldots,x_s\in W_{2^t+1,4}$ and $y_1,\ldots,y_m\in H^*(\gr_{2^t+1,4})\setminus W_{2^t+1,4}$. Since $\gr_{2^t+1,4}$ is simply connected, $|x_i|\ge2$ for all $i\in\{1,\ldots,s\}$, and according to (\ref{charrank}), $|y_i|\ge2^t$ for all $i\in\{1,\ldots,m\}$. Thus for the total degree $d$ of the class $x_1\cdots x_sy_1\cdots y_m$ we have
\[d\ge2s+2^tm\ge2(2^t+2^{t-2}-3-m)+2^tm=(2^t-2)m+2\cdot2^t+2^{t-1}-6.\]
By (\ref{clW}), $s\le2^t+2^{t-2}-5$, and so $s+m\ge2^t+2^{t-2}-3$ implies $m\ge2$. We conclude
\[d\ge2\cdot(2^t-2)+2\cdot2^t+2^{t-1}-6=4\cdot2^t+2^{t-1}-10>4\cdot2^t-12=\dim(\gr_{2^t+1,4}),\]
and consequently $x_1\cdots x_sy_1\cdots y_m=0$. This contradiction finishes the proof.
\end{proof}

From (\ref{eq:cat>1+cup}) we now get a lower bound for the Lusternik--Schnirelmann category of $\gr_{2^t+1,4}$.
\begin{corollary} For all integers $t\ge 3$ one has
    \[\cat(\gr_{2^t+1,4})\ge2^t+2^{t-2} - 3.\]
\end{corollary}

\section{Zero-divisor cup-length of $\gr_{n,k}$}
\label{sec:5}

This section is devoted to finding lower bounds for topological complexity of Grassmannians $\gr_{n,4}$ by studying zero-divisor cup-length of $H^*(\gr_{n,4})$ and its subalgebra $W_{n,4}$. 

Let us first note that for any $k\ge2$ the sequence $\zcl(W_{n,k})$ increases with $n$:
\begin{equation}\label{zcl(W_{n,k}) increases}
   \zcl(W_{n,k})\le\zcl(W_{n+1,k}).
\end{equation}
This was proved in \cite[Lemma 3.3]{CPR} for $k=3$, but virtually the same proof goes through for arbitrary $k$ (the essence of the proof is the fact $I_{n,k}\supseteq I_{n+1,k}$).

We now establish a connection between $\zcl(W_{n,k})$ and $\zcl(\gr_{n,k})$. Note first that, since we are working over a field, $W_{n,k}\subset H^*(\gr_{n,k})$ implies $W_{n,k}\otimes W_{n,k}\subset H^*(\gr_{n,k})\otimes H^*(\gr_{n,k})$. This means that 
every zero-divisor in $W_{n,k}\otimes W_{n,k}$ is a zero-divisor in $H^*(\gr_{n,k})\otimes H^*(\gr_{n,k})$ as well (see the following diagram).
\[\begin{tikzcd}[row sep = 1cm, column sep = 0.7cm]
    H^*(\gr_{n,k})\otimes H^*(\gr_{n,k})
    \arrow{r}{\cdot}
    &
    H^*(\gr_{n,k})\\
    W_{n,k}\otimes W_{n,k}
    \arrow[hook]{u}
    \arrow{r}{\cdot}
    &
    W_{n,k}
    \arrow[hook]{u}
\end{tikzcd}\]
A consequence of this remark is the inequality $\zcl(\gr_{n,k})\ge\zcl(W_{n,k})$. However, we can prove more than that.
\begin{proposition}\label{prop:zcl(Gr)>=1+zcl(W)}
    Let $n\ge2k\ge4$. Then
    \[\zcl(\gr_{n,k}) \ge 1 + \zcl(W_{n,k}).\]
\end{proposition}
\begin{proof}
 Let $z_1,z_2,\ldots,z_s\in W_{n,k}\otimes W_{n,k}$ be homogeneous zero-divisors of positive degree such that $z_1z_2\cdots z_s\ne0$ and $s=\zcl(W_{n,k})$, i.e., $\zcl(W_{n,k})$ is reached by the product $z_1z_2\cdots z_s$, which we denote by $z$. Let $m$ be the degree of $z$ (i.e., the sum of degrees of $z_1,z_2,\ldots,z_s$). So $z$ is an element of degree $m$ in $H^*(\gr_{n,k})\otimes H^*(\gr_{n,k})$, which means that
    \[z\in\bigoplus_{i = 0}^m H^i(\gr_{n,k}) \otimes H^{m-i}(\gr_{n,k}).\]
Let us write $z$ in the form $\sum_{i=0}^m\alpha_i$, where $\alpha_i\in H^i(\gr_{n,k}) \otimes H^{m-i}(\gr_{n,k})$. The dimension of the manifold $\gr_{n,k}$ is $k(n-k)$, and so $H^i(\gr_{n,k})=0$ for $i>k(n-k)$. Furthermore, note that $\alpha_i$ in fact belongs to $\big(H^i(\gr_{n,k})\cap W_{n,k}\big) \otimes \big(H^{m-i}(\gr_{n,k})\cap W_{n,k}\big)$ (because $z\in W_{n,k}\otimes W_{n,k}$), and it was proved in \cite[p.\ 1171]{Korbas} that $H^{k(n-k)}(\gr_{n,k})\cap W_{n,k}=0$. Therefore, the following implication holds:
\begin{equation}\label{alfai=0}
    i\ge k(n-k) \quad \Longrightarrow \quad \alpha_i=0.
\end{equation}

Since $z\ne0$ there exists $j\in\{0,1,\dots,m\}$ such that $\alpha_j\ne 0$ ((\ref{alfai=0}) implies $j<k(n-k)$). 

Next, let $\{b_1,b_2,\dots,b_r\}$ be a vector space basis for $H^j(\gr_{n,k})$. We can now write $\alpha_j$ in the form $\sum_{l=1}^r b_l\otimes v_l$, for some $v_l\in H^{m-j}(\gr_{n,k})$. Since $\alpha_j\ne0$, there exists $l_0\in\{1,\dots,r\}$ with the property $b_{l_0}\otimes v_{l_0}\ne0$. If we define a map $\varphi:H^j(\gr_{n,k})\rightarrow H^{k(n-k)}(\gr_{n,k})$ on the basis elements by
\[\varphi(b_l) = \begin{cases}
    c, & l = l_0\\
    0, & l \ne l_0
\end{cases},\]
where $c\in H^{k(n-k)}(\gr_{n,k})$ is the generator, then Poincar\'e duality gives us a class $a\in H^{k(n-k)-j}(\gr_{n,k})$ such that $\varphi$ is the multiplication with $a$. So for $1\le l \le r$, we have $ab_l \ne 0$ if and only if $l = l_0$. 

We now claim that $z(a)\cdot z_1z_2\cdots z_s=(1\otimes a+a\otimes1)\cdot z\ne0$ in $H^*(\gr_{n,k})\otimes H^*(\gr_{n,k})$, which will prove the proposition. The degree of this element is $k(n-k)-j+m$, and we have
\[(1\otimes a+a\otimes1)\cdot z=(1\otimes a+a\otimes1)\sum_{i = 0}^m\alpha_i =\sum_{i = 0}^m (1\otimes a)\alpha_i+\sum_{i = 0}^m (a\otimes 1)\alpha_i.\]
To verify that this is nonzero it will be enough to prove that the summand in $H^{k(n-k)}(\gr_{n,k})\otimes H^{m-j}(\gr_{n,k})$ is nonzero. However, the said summand is equal to $(a\otimes 1)\alpha_j$ (if $m\ge k(n-k)$, then this summand contains $(1\otimes a)\alpha_{k(n-k)}$ as well, but $\alpha_{k(n-k)}=0$ due to (\ref{alfai=0})), and we have
\[(a\otimes 1)\alpha_j= (a\otimes 1)\sum_{l=1}^r b_l\otimes v_l = \sum_{l=1}^r ab_l\otimes v_l = ab_{l_0}\otimes v_{l_0}\ne 0\]
(since we are working over a field, the fact $ab_{l_0}\otimes v_{l_0}\ne 0$ follows from $ab_{l_0}\ne0$ and $v_{l_0}\ne0$).
\end{proof}

Our aim now is to find a lower bound for $\zcl(W_{2^t-2,4})$, which will also be a lower bound for $\zcl(W_{n,4})$ for all $n\ge2^t-2$ (by (\ref{zcl(W_{n,k}) increases})). Together with Proposition \ref{prop:zcl(Gr)>=1+zcl(W)} and (\ref{eq:TC>1+zcl}), this will produce a lower bound for $\tc(\gr_{n,4})$. For that purpose we now make some preparatory remarks.

\medskip

We will need an algebra morphism $\psi:W_{2^t-1,3}\arr W_{2^t-2,4}$ (for $t\ge4$) that takes Stiefel--Whitney classes to corresponding Stiefel--Whitney classes, i.e., such that $\psi(\w_2)=\w_2$ and $\psi(\w_3)=\w_3$. Note that if such a morphism exists, then it is unique, because $W_{2^t-1,3}$ is generated by $\w_2$ and $\w_3$. Here is how it can be constructed.
The kernel of the composition of the inclusion $\F_2[w_2,w_3]\subset\F_2[w_2,w_3,w_4]$ with the natural projection $\F_2[w_2,w_3,w_4]\arr\F_2[w_2,w_3,w_4]/I_{2^t-2,4}$ is $I_{2^t-2,4}\cap\F[w_2,w_3]$. By (\ref{I_2^t=I_2^t-1}), Lemma \ref{utapanje} and (\ref{In,k opada}), $I_{2^t-1,3}=I_{2^t,3}=I_{2^t,4}\cap\F[w_2,w_3]\subseteq I_{2^t-2,4}\cap\F[w_2,w_3]$, and so this composition induces the algebra morphism $\Psi$ presented in the following diagram.
\[
\begin{tikzcd}[row sep = 1cm, column sep = 1cm]
    \F_2[w_2,w_3]
    \arrow[hook]{r}
    \arrow[twoheadrightarrow]{d}
    &
    \F_2[w_2,w_3,w_4]
    \arrow[twoheadrightarrow]{d}
    \\
    \F_2[w_2,w_3]/I_{2^t-1,3}
    \arrow[dashed]{r}{\Psi}
    &
    \F_2[w_2,w_3,w_4]/I_{2^t-2,4}
\end{tikzcd}
\]
Having in mind the isomorphism (\ref{eq:imp*2}) we obtain the morphism $\psi:W_{2^t-1,3}\arr W_{2^t-2,4}$ with the desired property.

Now let $B$ be the vector space basis for $W_{2^t-1,3}$ obtained from the Gr\"obner basis $F$ for $I_{2^t-1,3}=I_{2^t,3}$ using Theorem \ref{additive basis} and the isomorphism (\ref{eq:imp*2}). So $B$ consists of the cohomology classes $\w_2^b\w_3^c\in W_{2^t-1,3}$ such that the corresponding monomial $w_2^bw_3^c$ is not divisible by any of the monomials from $\lm(F)$. Let $B_{2^t-2}$ be the vector space basis for $W_{2^t-2,4}$ obtained in the same way from the Gr\"obner basis $F_{2^t-2}$ for $I_{2^t-2,4}$. Since $\lm(F_{2^t-2})=\lm(F)\cup\big\{w_4^{2^{t-2}-1},w_3w_4^{2^{t-2}-2}\big\}$ and $\psi(\w_2^b\w_3^c)=\w_2^b\w_3^c$, we see that for any $\sigma\in B$ and any $d\in\{0,1,\ldots,2^{t-2}-3\}$ one has 
\begin{equation}\label{bazni}
\psi(\sigma)\w_4^d\in B_{2^t-2}.
\end{equation}
Moreover, for $\sigma_1,\sigma_2\in B$ and $d_1,d_2\in\{0,1,\ldots,2^{t-2}-3\}$,
\begin{equation}\label{razliciti_bazni}
    (\sigma_1,d_1)\ne(\sigma_2,d_2) \quad \Longrightarrow \quad \psi(\sigma_1)\w_4^{d_1}\ne\psi(\sigma_2)\w_4^{d_2}.
\end{equation}

In the next lemma we prove the crucial fact for obtaining the announced lower bound for $\zcl(W_{2^t-2,4})$.
\begin{lemma}\label{psi_tenzor_psi}
  Let $t\ge5$ be an integer and $z\in W_{2^t-1,3}\otimes W_{2^t-1,3}$ an arbitrary element. If $z\ne0$ in $W_{2^t-1,3}\otimes W_{2^t-1,3}$, then 
  \[(\psi\otimes\psi)(z)\cdot\sum_{d=2}^{2^{t-2}-3}\w_4^d\otimes\w_4^{2^{t-2}-1-d}\ne0 \quad \mbox{ in } W_{2^t-2,4}\otimes W_{2^t-2,4}.\]
\end{lemma}
\begin{proof}
    Since $B$ is a basis for $W_{2^t-1,3}$, the set $\{\sigma\otimes\tau\mid\sigma,\tau\in B\}$ is a basis for $W_{2^t-1,3}\otimes W_{2^t-1,3}$. Similarly, simple tensors of elements from $B_{2^t-2}$ constitute a basis for $W_{2^t-2,4}\otimes W_{2^t-2,4}$. Write $z$ as a nonempty sum $\sum_{i=1}^r\sigma_i\otimes\tau_i$ of distinct basis elements. Then
    \begin{align*}
      (\psi\otimes\psi)(z)\cdot\sum_{d=2}^{2^{t-2}-3}\w_4^d\otimes\w_4^{2^{t-2}-1-d}&=\sum_{i=1}^r\psi(\sigma_i)\otimes\psi(\tau_i)\cdot\sum_{d=2}^{2^{t-2}-3}\w_4^d\otimes\w_4^{2^{t-2}-1-d}\\
      &=\sum_{i=1}^r\sum_{d=2}^{2^{t-2}-3}\psi(\sigma_i)\w_4^d\otimes\psi(\tau_i)\w_4^{2^{t-2}-1-d},
    \end{align*}
    and according to (\ref{bazni}) and (\ref{razliciti_bazni}), this is a nonempty sum of distinct basis elements in $W_{2^t-2,4}\otimes W_{2^t-2,4}$. 
\end{proof}

Let us now indicate the idea of constructing a "long" nontrivial product of zero-divisors in $W_{2^t-2,4}\otimes W_{2^t-2,4}$ (i.e., obtaining a lower bound for $\zcl(W_{2^t-2,4})$). We will take $z\in W_{2^t-1,3}\otimes W_{2^t-1,3}$ to be a product of zero-divisors that realizes $\zcl(W_{2^t-1,3})$ (known from \cite{CPR}), then we will note that $(\psi\otimes\psi)(z)$ is the product of corresponding zero-divisors in $W_{2^t-2,4}\otimes W_{2^t-2,4}$, multiply it with $z(\w_4)^{2^{t-2}-1}$, and prove that the obtained product is nonzero. The zero-divisor $z(\w_4)^{2^{t-2}-1}$ is equal to the sum $\sum_{d=0}^{2^{t-2}-1}\w_4^d\otimes\w_4^{2^{t-2}-1-d}$, and after applying Lemma \ref{psi_tenzor_psi}, to prove the claim it will suffice to verify that for $d\in\{0,1,2^{t-2}-2,2^{t-2}-1\}$ one has $(\psi\otimes\psi)(z)\cdot(\w_4^d\otimes\w_4^{2^{t-2}-1-d})=0$. The identities from the following lemma will play the crucial role in that verification.

\begin{lemma}\label{identiteti:2^t-2}
    For all integers $t\ge4$ the following equalities hold in $W_{2^t-2,4}$:
    \begin{itemize}
        \item[(a)] $\w_2^{2^{t-2}}\w_3^{2^{t-2}-1}=0$, 
        \item[(b)] $\w_3^{2^{t-2}-1}\w_4^{2^{t-2}-2}=0$, 
        \item[(c)] $\w_2^{2^{t-1} + 2^{t-2}
        }\w_4^{2^{t-2}-2} = 0$.
    \end{itemize}
\end{lemma}

\begin{proof}
(a) We want to prove that $w_2^{2^{t-2}}w_3^{2^{t-2}-1}\in I_{2^t-2,4}$. By \cite[Proposition 2.2(c)]{CP-cl} and (\ref{g_rinIn,k}), $w_2^{2^{t-2}}w_3^{2^{t-2}-1}=g_{2^t+2^{t-2}-3}^{(3)}\in I_{2^t,3}$, and Lemma \ref{utapanje}, along with (\ref{In,k opada}), implies $w_2^{2^{t-2}}w_3^{2^{t-2}-1}\in I_{2^t,4}\subseteq I_{2^t-2,4}$.
  
    \medskip
    
 (b) Since $g_r\in I_{2^t-2,4}$ for all $r\ge2^t-5$, the identity $\w_3^{2^{t-2}-1}\w_4^{2^{t-2}-2}=0$ (i.e., $w_3^{2^{t-2}-1}w_4^{2^{t-2}-2}\in I_{2^t-2,4}$) is a consequence of the following relation in $\F_2[w_2,w_3,w_4]$: 
    \begin{equation}\label{eq:rnd1-lema57}
        w_3^{2^{t-2} - 1}w_4^{2^{t-2} - 2} = w_3^{2^{t-2} - 2}g_{2^t - 5} + \sum_{i = 2}^{t-2}w_3^{2^{t-2} - 2^i}w_4^{2^{i-1}-2}  g_{2^t -3 + 2^i}.
        \end{equation}
This relation actually holds for $t=3$ as well (as usual, the empty sum is understood to be zero), and we are going to prove it by induction on $t\ge3$. 

For $t=3$, (\ref{eq:rnd1-lema57}) simplifies to $w_3=g_3$, which is obviously true (see (\ref{eq_explicitly_g})).
    
Now let $t\ge4$, and assume that the claim holds for $t-1$, that is
    \[w_3^{2^{t-3} - 1}w_4^{2^{t-3} - 2} = w_3^{2^{t-3} - 2}g_{2^{t-1} - 5} + \sum_{i = 2}^{t-3}w_3^{2^{t-3} - 2^i}w_4^{2^{i-1}-2} g_{2^{t-1} -3 + 2^i}.\]
    Squaring both sides and multiplying them by $w_3w_4^2$ yields
    \[w_3^{2^{t-2} - 1}w_4^{2^{t-2} - 2} = w_3^{2^{t-2} - 3}w_4^2\big(g_{2^{t-1} - 5}\big)^2 + \sum_{i = 2}^{t-3}w_3^{2^{t-2} - 2^{i+1}+1}w_4^{2^{i}-2}\big(g_{2^{t-1} -3 + 2^i}\big)^2.\]
    By (\ref{w_3kvadrat}), $w_3(g_{2^{t-1}-5})^2 = g_{2^t -7}$ and $w_3(g_{2^{t-1} -3 + 2^i})^2 = g_{2^t-3 + 2^{i+1}}$, so after applying this, and shifting the summation index, we get
     \begin{equation}\label{eq: prosla}w_3^{2^{t-2} - 1}w_4^{2^{t-2} - 2} = w_3^{2^{t-2} - 4}w_4^2g_{2^t-7} + \sum_{i = 3}^{t-2}w_3^{2^{t-2} - 2^{i}}w_4^{2^{i-1}-2} g_{2^{t} -3 + 2^i.}
     \end{equation}
    Finally, by equation (\ref{eq_rec_g_gen}) for $k=4$, $r = 2^t+1$ and $s= 1$, we have
    \[g_{2^t +1} = w_2^2 g_{2^t-3} + w_3^2g_{2^t - 5} + w_4^2g_{2^t-7} = w_3^2g_{2^t - 5} + w_4^2g_{2^t-7}.\]
    Therefore, we can substitute $w_3^2g_{2^t - 5}+g_{2^t +1}$ for $w_4^2g_{2^t-7}$ in (\ref{eq: prosla}), and after doing so, we get \eqref{eq:rnd1-lema57}. This completes the induction step.

    \medskip
    
(c) It was proved in \cite[Corollary 3.8]{CPR} that 
\[w_2^{2^{t-1}+2^{t-2}}+\sum_{i=1}^{t-3}w_2^{2^i+2^{i-1}}w_3^{2^{t-1}-2^i}\in I_{2^t+2^{t-2}+2,3}.\]
On the other hand, we know that $I_{2^t+2^{t-2}+2,3}\subseteq I_{2^t,3}\subseteq I_{2^t,4}\subseteq I_{2^t-2,4}$ (by (\ref{In,k opada}) and Lemma \ref{utapanje}). Therefore, in $W_{2^t-2,4}$ we have $\w_2^{2^{t-1}+2^{t-2}}=\sum_{i=1}^{t-3}\w_2^{2^i+2^{i-1}}\w_3^{2^{t-1}-2^i}$, which means that
\[\w_2^{2^{t-1}+2^{t-2}}\w_4^{2^{t-2}-2}=\sum_{i=1}^{t-3}\w_2^{2^i+2^{i-1}}\w_3^{2^{t-1}-2^i}\w_4^{2^{t-2}-2}=0,\]
by part (b), since $2^{t-1}-2^i\ge2^{t-1}-2^{t-3}>2^{t-2} - 1$ for all $i\in\{1,\ldots,t-3\}$.
\end{proof}

We are finally ready to establish lower bounds for $\zcl(W_{n,4})$.

\begin{proposition}\label{zcl:dostizanje 2^t-2}
    For all integers $t\ge5$ and $n\ge2^t-2$ one has 
    \[\zcl(W_{n,4})\ge2^t+2^{t-1}+2^{t-2}-5.\]
\end{proposition}
\begin{proof}
    We have already outlined the idea of the proof prior to Lemma \ref{identiteti:2^t-2}. First of all, according to (\ref{zcl(W_{n,k}) increases}), it is enough to prove $\zcl(W_{2^t-2,4})\ge2^t+2^{t-1}+2^{t-2}-5$. This inequality will follow from the fact $z(\w_2)^{2^t-1}z(\w_3)^{2^{t-1}-3}z(\w_4)^{2^{t-2}-1}\ne0$ in $W_{2^t-2,4}\otimes W_{2^t-2,4}$, which we now prove.

    \medskip

    Let $z:=z(\w_2)^{2^t-1}z(\w_3)^{2^{t-1}-3}\in W_{2^t-1,3}\otimes W_{2^t-1,3}$. It was proved in \cite[Proposition 3.9 and Lemma 3.4]{CPR} that $z\ne0$. We have constructed the algebra morphism $\psi:W_{2^t-1,3}\arr W_{2^t-2,4}$ with the properties $\psi(\w_2)=\w_2$ and $\psi(\w_3)=\w_3$. Therefore,
    \[(\psi\otimes\psi)\big(z(\w_2)\big)=(\psi\otimes\psi)\big(1\otimes\w_2+\w_2\otimes1\big)=1\otimes\w_2+\w_2\otimes1=z(\w_2),\]
    and similarly, $(\psi\otimes\psi)\big(z(\w_3)\big)=z(\w_3)$. The tensor product $\psi\otimes\psi:W_{2^t-1,3}\otimes W_{2^t-1,3}\arr W_{2^t-2,4}\otimes W_{2^t-2,4}$ is also an algebra morphism, and so
    \begin{equation}\label{(psi_tenzor_psi)(z)}
    (\psi\otimes\psi)(z)=z(\w_2)^{2^t-1}z(\w_3)^{2^{t-1}-3}.
    \end{equation}
On the other, by binomial formula the zero-divisor $z(\w_4)^{2^{t-2}-1}$ is equal to
\[(1\otimes\w_4+\w_4\otimes1)^{2^{t-2}-1}=\sum_{d=0}^{2^{t-2}-1}\binom{2^{t-2}-1}{ d}\w_4^d\otimes\w_4^{2^{t-2}-1-d}=\sum_{d=0}^{2^{t-2}-1}\w_4^d\otimes\w_4^{2^{t-2}-1-d}\]
(since $\binom{2^{t-2}-1}{d}\equiv1\pmod2$ if $0\le d\le2^{t-2}-1$), and therefore,
\[
   z(\w_2)^{2^t-1}z(\w_3)^{2^{t-1}-3}z(\w_4)^{2^{t-2}-1}=(\psi\otimes\psi)(z)\cdot\sum_{d=0}^{2^{t-2}-1}\w_4^d\otimes\w_4^{2^{t-2}-1-d}.
\]
Since $z\ne0$, by Lemma \ref{psi_tenzor_psi} we know that 
\[(\psi\otimes\psi)(z)\cdot\sum_{d=2}^{2^{t-2}-3}\w_4^d\otimes\w_4^{2^{t-2}-1-d}\ne0,\] and so it is now sufficient to verify that for $d\in\{0,1,2^{t-2}-2,2^{t-2}-1\}$ the element $x_d:=(\psi\otimes\psi)(z)\cdot(\w_4^d\otimes\w_4^{2^{t-2}-1-d})\in W_{2^t-2,4}\otimes W_{2^t-2,4}$ vanishes. 

Using (\ref{(psi_tenzor_psi)(z)}) we calculate:
\begin{align*}
    x_d&=z(\w_2)^{2^t-1}z(\w_3)^{2^{t-1}-3}\big(\w_4^d\otimes\w_4^{2^{t-2}-1-d}\big)\\
    &=\sum_{b=0}^{2^t-1}\w_2^b\otimes\w_2^{2^t-1-b}\cdot\sum_{c=0}^{2^{t-1}-3}\binom{2^{t-1}-3}{c}\w_3^c\otimes\w_3^{2^{t-1}-3-c}\cdot\big(\w_4^d\otimes\w_4^{2^{t-2}-1-d}\big)\\
    &=\sum_{b=0}^{2^t-1}\sum_{c=0}^{2^{t-1}-3}\binom{2^{t-1}-3}{c}\w_2^b\w_3^c\w_4^d\otimes\w_2^{2^t-1-b}\w_3^{2^{t-1}-3-c}\w_4^{2^{t-2}-1-d}.
\end{align*}
For $0\le d\le1$ let us show that in fact every summand of $x_d$ is equal to zero. Pick an arbitrary summand $s_d(b,c)=\w_2^b \w_3^c \w_4^d \otimes \w_2^{2^t - 1- b} \w_3^{2^{t-1} - 3 - c} \w_4^{2^{t-2} - 1-d}$. If $b \ge 2^{t-2}$ and $c \ge 2^{t-2} -1$, then the first coordinate of $s_d(b,c)$ is zero by Lemma \ref{identiteti:2^t-2}(a). Otherwise, if either $b \le 2^{t-2} - 1$ or $c \le 2^{t-2} - 2$, then the second coordinate vanishes by parts (b) and (c) of the same lemma. We conclude $x_d=0$.

It is not hard to check that $x_{2^{t-2}-1-d}=T(x_d)$, where $T:W_{2^t-2,4}\otimes W_{2^t-2,4}\arr W_{2^t-2,4}\otimes W_{2^t-2,4}$ interchanges the coordinates (it is defined on simple tensors by $T(\sigma\otimes\tau)=\tau\otimes\sigma$). Therefore, $x_d=0$ holds for $2^{t-2}-2\le d\le2^{t-2}-1$ as well. This concludes the proof.
\end{proof}

\begin{remark}
    We believe that the inequality from the previous proposition is an equality for $n\in\{2^t-2,2^t-1,2^t,2^t+1\}$. Computer calculations confirmed our assumption for $5\le t\le 20$.
\end{remark}

\begin{remark}\label{remark2}
    The proposition does not hold for $t=4$. It is readily seen from its proof that the zero-divisor $z(\w_2)^{2^t-1}z(\w_3)^{2^{t-1}-3}z(\w_4)^{2^{t-2}-1}=z(\w_2)^{15}z(\w_3)^5z(\w_4)^3$ vanishes in $W_{14,4}\otimes W_{14,4}$. Furthermore, it can be proved that $\zcl(W_{14,4})=21$. 

    However, $z(\w_2)^{15}z(\w_3)^5z(\w_4)^3\ne0$ in $W_{15,4}\otimes W_{15,4}$, so Proposition \ref{zcl:dostizanje 2^t-2} does hold for $t=4$ and $n\ge15$. Moreover, $\zcl(W_{15,4})=\zcl(W_{16,4})=\zcl(W_{17,4})=23$.

    Let us also remark that one can show that $\zcl(W_{8,4})=\zcl(W_{9,4})=8$.

    We have checked all these listed results by using computer software SageMath \cite{Sage}.
\end{remark}

Using Proposition \ref{prop:zcl(Gr)>=1+zcl(W)} we immediately get the following corollary.

\begin{corollary}\label{prop:finale}
    For all integers $t\ge5$ and $n\ge2^t-2$ one has 
    \[\zcl(\gr_{n,4})\ge2^t+2^{t-1}+2^{t-2}-4.\]
\end{corollary}

Finally, (\ref{eq:TC>1+zcl}) produces lower bounds for the topological complexity of oriented Grassmann manifolds $\gr_{n,4}$.
\begin{theorem} For all integers $t\ge5$ and $n\ge2^t-2$ one has 
    \[\tc(\gr_{n,4})\ge2^t+2^{t-1}+2^{t-2}-3.\]
\end{theorem}

\begin{remark}
    This theorem covers the cases $n\ge30$. For $n\le29$ in the same way one can obtain lower bounds for $\tc(\gr_{n,4})$ from the results listed in Remark \ref{remark2}.
\end{remark}

\bibliographystyle{amsplain}

\end{document}